\documentclass[11pt]{article}
\usepackage[T1]{fontenc}
\usepackage{lmodern}
\usepackage[margin=1in]{geometry}
\usepackage{amsmath,amssymb,amsthm,mathtools}
\usepackage{microtype}
\usepackage[hidelinks]{hyperref}
\hypersetup{pdftitle={Dimensions of permanental varieties in arbitrary size},
  pdfauthor={Ying Xie},pdfcreator={},pdfproducer={}}
\ifdefined\pdfinfoomitdate\pdfinfoomitdate=1\fi
\ifdefined\pdfsuppressptexinfo\pdfsuppressptexinfo=15\fi
\ifdefined\pdftrailerid\pdftrailerid{}\fi
\newtheorem{theorem}{Theorem}[section]

\newtheorem{lemma}[theorem]{Lemma}
\newtheorem{corollary}[theorem]{Corollary}
\theoremstyle{remark}
\newtheorem{remark}[theorem]{Remark}
\DeclareMathOperator{\per}{per}
\DeclareMathOperator{\codim}{codim}
\DeclareMathOperator{\Crit}{Crit}

\DeclareMathOperator{\htop}{in}
\DeclareMathOperator{\Spec}{Spec}
\newcommand{\A}{\mathbb A}
\newcommand{\Gm}{\mathbb G_m}
\newcommand{\kk}{K}

\newcommand{\smallset}[2]{#1\setminus\{#2\}}
\numberwithin{equation}{section}
\title{Dimensions of permanental varieties in arbitrary size}
\author{Ying Xie\\[3pt]
  \small Kennesaw State University\\
  \small\texttt{yxie2@kennesaw.edu}}
\date{}

\begin{document}
\maketitle
\begin{abstract}
Let $I_h(a,b)$ be the ideal generated by the $h\times h$ permanents
of a generic $a\times b$ matrix over a field of characteristic different
from two. We prove that
$\dim \kk[X]/I_h(a,b)=\max(a,b)(h-1)$ whenever
$1\leq h\leq\min(a,b)$ and $h<\max(a,b)$.
Consequently, the critical locus of the permanent of order $n$ has
codimension $2n$ for every $n\geq2$. We also prove that the maximal
permanents of a generic $k\times(k+1)$ matrix generate a geometrically
reduced complete intersection. The dimension argument uses a pivot
expansion and cubic relations in the border variables. Two successive
comparisons with highest-degree equations reduce the relevant fibers
to a coordinate-support calculation. Reducedness follows from a
separate regular-sequence argument and a dimension estimate excluding
components over the exceptional locus.
\end{abstract}

\section{Introduction}

For a matrix $X=(x_{ij})$ of independent variables, the permanent of a
square submatrix is the sum of its transversal monomials, all with
coefficient one. Write $I_h(a,b)\subset\kk[x_{ij}]$ for the ideal of
all $h\times h$ permanents of a generic $a\times b$ matrix, and write
$V_h(a,b)$ for its zero locus over an algebraic closure. All dimensions
are affine dimensions. In particular, codimension is measured in the
full matrix space, rather than in a hypersurface.

Boralevi, Carlini, Micha\l ek, and Ventura studied the codimensions of
maximal permanental varieties and of the singular locus of the
permanental hypersurface \cite{BCMV}. Their Conjecture~3.4 asserts that
the $k+1$ maximal permanents of a generic $k\times(k+1)$ matrix define
a complete intersection in characteristic zero. Their Theorem~3.18
establishes the expected codimension for maximal permanents with
$k\leq5$. Earlier work of Kirkup treats minimal primes and the
$3\times4$ case in detail \cite[Section~7]{Kirkup}.

The following result gives the dimension for every proper subpermanent
size, including maximal permanents of nonsquare matrices.

\begin{theorem}\label{thm:dimension}
Let $\kk$ be a field with $\operatorname{char}\kk\neq2$. If
$1\leq h\leq\min(a,b)$ and $h<\max(a,b)$, then
\begin{equation}\label{eq:main-dimension}
 \dim\kk[X]/I_h(a,b)=\max(a,b)(h-1).
\end{equation}
Equivalently,
\[
 \operatorname{ht} I_h(a,b)
   =\max(a,b)\bigl(\min(a,b)-h+1\bigr).
\]
\end{theorem}

The condition $h<\max(a,b)$ excludes the full permanent of a square
matrix, which defines a hypersurface. Taking $a=b=n$ and $h=n-1$
gives the exact critical-locus dimension.

\begin{samepage}
\begin{corollary}\label{cor:critical-intro}
For $n\geq2$ and $\operatorname{char}\kk\neq2$,
\[
 \codim_{\A^{n^2}}\Crit(\per_n)=2n.
\]
The same equality holds for the singular locus of
$\{\per_n=0\}$, with codimension still measured in $\A^{n^2}$.
\end{corollary}
\end{samepage}

The dimension calculation alone does not imply reducedness of the
ideals. For maximal permanents of a matrix with one more column than
rows, a further argument gives the following scheme-theoretic result.

\begin{theorem}\label{thm:radical-intro}
For every $k\geq1$ and every field $\kk$ of characteristic different
from two, $I_k(k,k+1)$ is a geometrically radical complete-intersection
ideal of height $k+1$.
\end{theorem}

Here geometrically radical means that the ideal remains radical after
every field extension. In particular, Theorem~\ref{thm:radical-intro}
proves the complete-intersection assertion of \cite[Conjecture~3.4]{BCMV}
and also gives reducedness in every characteristic other than two.

The prime-size preprint \cite{XiePrime} proves a Frobenius-splitting
statement by an explicit trace calculation in the relevant prime
characteristic. The present paper gives a separate dimension proof
valid at every size; no recursion from a prime size is used.
Its conclusions about dimension and reducedness do not supply a
Frobenius splitting. The arithmetic distinction is already visible
for the $3\times4$ ideal, whose Frobenius behavior is studied in
\cite{XieThree}.

The proof of Theorem~\ref{thm:dimension} is in
Sections~\ref{sec:pivot}--\ref{sec:dimension}. After normalizing a
nonzero matrix entry, one first takes highest-degree parts in the
interior block. Cubic relations in the border survive this step.
A second highest-degree comparison, made after fixing a smaller
nonzero subpermanent, bounds the border fibers. Section~\ref{sec:radical}
proves reducedness, with the required filtered-algebra lemma supplied
in Section~\ref{sec:filtered}. Section~\ref{sec:consequences} records
consequences and the characteristic-two boundary.

\section{A pivot identity and two elementary lemmas}\label{sec:pivot}

We work over an algebraically closed field until otherwise stated.
For row and column sets $A,B$ of the same size, set
\[
 p_{A,B}=\per Y[A,B],\qquad p_{\varnothing,\varnothing}=1.
\]
Whenever a pivot entry is nonzero, multiplying its row by the inverse
of that entry identifies the chart with $\Gm$ times a chart whose
pivot is one. Row and column permutations allow us to place the pivot
in the lower-right corner. The normalized matrix has the form
\begin{equation}\label{eq:pivot-matrix}
 X=\begin{pmatrix}Y&u\\ v^{\mathsf T}&1\end{pmatrix}.
\end{equation}
These operations preserve the vanishing of all permanents of a fixed
size. For $|A|=|B|$, denote by $P_{A,B}$ the permanent obtained by
adjoining the pivot row and column to $Y[A,B]$. Direct expansion gives
\begin{equation}\label{eq:border}
 P_{A,B}=p_{A,B}
   +\sum_{i\in A,\,j\in B}u_i v_j
                  p_{\smallset A i,\smallset B j}.
\end{equation}
If $|A|=h-1$ and $|B|=h$, let $R_{A,B}$ be the $h\times h$
permanent using rows $A$ and the pivot row, and columns $B$, without
the pivot column.

\begin{lemma}[Cubic border relation]\label{lem:cubic}
For these sets $A,B$,
\begin{equation}\label{eq:cubic-identity}
 \sum_{j\in B}v_jP_{A,\smallset B j}-R_{A,B}
 =2\sum_{\substack{i\in A\\ \{j,\ell\}\subset B}}
       u_i v_jv_\ell\,
       p_{\smallset A i,B\setminus\{j,\ell\}}.
\end{equation}
The sum on the right runs over unordered two-element column sets.
There is a transposed identity with two $u$ variables and one $v$
variable.
\end{lemma}

\begin{proof}
Expand every $P$ using \eqref{eq:border}. Its first terms sum to
$\sum_{j\in B}v_jp_{A,\smallset B j}=R_{A,B}$.
In the remaining sum, the two distinct columns $j,\ell$ occur in both
orders. Their two contributions are equal, giving the factor two.
Transposing the matrix gives the second identity.
\end{proof}

\begin{lemma}[Highest-degree comparison]\label{lem:dimension}
Let $f_1,\ldots,f_s\in\kk[y_1,\ldots,y_d,z_1,\ldots,z_e]$.
Let $q_i$ be the nonzero part of $f_i$ of highest total degree in the
$y$ variables, with zero polynomials omitted. Then
\[
 \dim V(f_1,\ldots,f_s)\leq\dim V(q_1,\ldots,q_s).
\]
The variables $z$ have weight zero and need not be specialized.
\end{lemma}

\begin{proof}
Let $Z$ be an irreducible component of $V(f_1,\ldots,f_s)$ of
dimension $d_Z$. In $\A^{d+e}\times\A^1$, take the closure $\mathcal Z$
of
\[
 \{(t y,z,t):(y,z)\in Z,\ t\neq0\}.
\]
It is irreducible of dimension $d_Z+1$. Its fiber at $t=0$ is nonempty:
scaling any fixed point of $Z$ gives a curve with a limit at $t=0$.
The coordinate $t$ is neither zero nor a unit on $\mathcal Z$.
The principal ideal theorem and the dimension formula for affine
domains over a field therefore give
$\dim(\mathcal Z\cap\{t=0\})\geq d_Z$.
If $d_i=\deg_y f_i$, the polynomial
$t^{d_i}f_i(y/t,z)$ vanishes on $\mathcal Z$ and restricts to $q_i$
at $t=0$. Thus that fiber lies in $V(q_1,\ldots,q_s)$.
Taking the maximum over $Z$ proves the assertion. This is the usual
weighted homogenization proof of the filtered dimension comparison.
\end{proof}

\begin{lemma}[The border support model]\label{lem:support}
For $k\geq2$, the zero locus in $\A^{2k}$ of
\begin{equation}\label{eq:support}
 u_i v_jv_\ell=0\quad(j\neq\ell),\qquad
 u_i u_{i'}v_j=0\quad(i\neq i')
\end{equation}
has dimension $k$.
\end{lemma}

\begin{proof}
If either $u=0$ or $v=0$, the other vector is arbitrary, giving two
$k$-dimensional coordinate spaces. If both vectors are nonzero, the
first family forces $v$ to have at most one nonzero coordinate, and
the second forces the same for $u$. This part is contained in the
union of the coordinate planes on one $u$ coordinate and one $v$
coordinate. The whole locus is therefore a finite union of spaces
of dimensions $k$ and two, with maximum $k$.
\end{proof}

\section{The dimension theorem}\label{sec:dimension}

\begin{proof}[Proof of Theorem~\ref{thm:dimension} for square matrices]
We prove simultaneously for all $1\leq h<n$ that
\begin{equation}\label{eq:square}
 \dim V_h(n,n)=n(h-1).
\end{equation}
The case $h=1$ is immediate, and starts the induction on $n$.
Suppose $2\leq h<n$, set $m=n-1$, and use the normalized pivot chart
\eqref{eq:pivot-matrix}, with $Y$ of size $m\times m$.
Put
\[
 r=h-2,\qquad k=m-r=n-h+1\geq2.
\]
The normalized ambient space has dimension $m^2+2m$.

Give entries of $Y$ weight one and all border variables weight zero.
For $|A|=|B|=h-1$, the highest $Y$-degree part of $P_{A,B}$ is
$p_{A,B}$. Since two is invertible, Lemma~\ref{lem:cubic} puts the
following polynomials in the normalized $h$-permanental ideal:
\begin{equation}\label{eq:C}
 C_{A,B}=
 \sum_{\substack{i\in A\\\{j,\ell\}\subset B}}
 u_i v_jv_\ell\,p_{\smallset A i,B\setminus\{j,\ell\}},
 \qquad |A|=h-1,\quad |B|=h,
\end{equation}
together with their transposes. They are homogeneous of $Y$-degree
$h-2$. Lemma~\ref{lem:dimension}, applied to these polynomials and
the bordered permanents, bounds the dimension of the normalized
chart by the dimension of the locus $W$ defined by
\begin{equation}\label{eq:W}
 Y\in V_{h-1}(m,m),\qquad C_{A,B}=0,
 \qquad C^{\mathsf T}_{A,B}=0.
\end{equation}
Only inclusion in the full ideal of highest parts is needed here;
we do not assert that the displayed equations generate that ideal.

First consider the open part of the base where some $r\times r$
permanent $c=p_{R,T}$ is nonzero. When $r=0$, use the empty permanent
$c=1$, and this is the entire base. Fix $Y$ in this open set, and
fix the $2r$ border coordinates $u_i$ for $i\in R$ and $v_j$ for
$j\in T$. There remain $k$ coordinates of each border vector.
For $i\notin R$ and distinct $j,\ell\notin T$, choose
$A=R\cup\{i\}$ and $B=T\cup\{j,\ell\}$ in \eqref{eq:C}.
Its highest degree in the remaining border coordinates is exactly
\begin{equation}\label{eq:second-top}
 c\,u_i v_jv_\ell.
\end{equation}
Indeed, any other summand uses at least one of the fixed border
coordinates and has degree at most two in the remaining ones.
The transposed equations similarly have highest parts
$c\,u_i u_{i'}v_j$.

A second application of Lemma~\ref{lem:dimension}, now in the
remaining border variables, and Lemma~\ref{lem:support} bound every
such fiber by $k$. Letting the fixed $2r$ coordinates vary gives
\begin{equation}\label{eq:fiber-bound}
 \dim W_Y\leq2r+k=2m-k.
\end{equation}
This is a bound for every $Y$ with $c\neq0$, not just a generic
fiber. The finitely many choices of $R,T$ cover the open part under
consideration. By induction the base $V_{h-1}(m,m)$ has codimension
$mk$ in $\A^{m^2}$. Consequently this part of $W$ has codimension
at least
\[
 mk+k=nk
\]
in the normalized ambient space.

If $h\geq3$, the remaining base is $V_{h-2}(m,m)$.
By induction its codimension is $m(k+1)$. Even allowing all $2m$
border coordinates, the part of $W$ above it has codimension at least
$m(k+1)$. This suffices because
\begin{equation}\label{eq:exceptional-dimension}
 m(k+1)-nk=m-k=h-2\geq0.
\end{equation}
For $h=2$ there is no remaining base. We have proved that the
normalized chart has dimension at most $m^2+2m-nk$.
Restoring the independent nonzero pivot coordinate adds one dimension,
so every nonzero-entry chart of $V_h(n,n)$ has dimension at most
$n^2-nk=n(h-1)$. These charts cover all points except the origin.

For the reverse inequality, matrices supported on any chosen $h-1$
rows form an $n(h-1)$-dimensional linear space in $V_h(n,n)$.
This proves \eqref{eq:square} and completes the induction.
\end{proof}

\begin{proof}[Rectangular matrices and arbitrary ground fields]
By transposition suppose $a\leq b$. If $h<b$, adjoining $b-a$ zero
rows embeds $V_h(a,b)$ as a closed subvariety of $V_h(b,b)$.
Every $h$-permanent involving an adjoined row vanishes, and the
others are exactly the original equations. The square case gives
$\dim V_h(a,b)\leq b(h-1)$. Matrices supported on $h-1$ rows give
the opposite inequality.

For a field that is not algebraically closed, Krull dimension of a
finitely generated algebra is unchanged by extension to an algebraic
closure. The equality and its height formulation therefore hold over
every field of characteristic different from two.
\end{proof}

\section{Lifting a radical ideal from highest parts}\label{sec:filtered}

The dimension comparison used above requires no equality of initial
ideals. Reducedness does require that equality. We record a sufficient
condition, including the case of variables of weight zero.

\begin{lemma}\label{lem:filtered}
Let $R$ be a polynomial ring over a field, equipped with nonnegative
integer weights on its variables. For $f_1,\ldots,f_s\in R$, let
$q_i=\htop(f_i)$ be their highest-weight parts. If
$q_1,\ldots,q_s$ form a regular sequence, then
\begin{equation}\label{eq:initial}
 \htop\bigl((f_1,\ldots,f_s)\bigr)=(q_1,\ldots,q_s).
\end{equation}
If $(q_1,\ldots,q_s)$ is also radical, then so is
$(f_1,\ldots,f_s)$.
\end{lemma}

\begin{proof}
Set $d_i=\deg f_i$ for the weighted degree. For a representation
$F=\sum_i a_if_i$, let $D=\max_i(\deg a_i+d_i)$.
If $D>\deg F$, the weight-$D$ terms give a homogeneous syzygy
\[
 \sum_i a_i^{[D-d_i]}q_i=0,
\]
where the bracket denotes a homogeneous component. Since the $q_i$
form a regular sequence, their first syzygies are generated by the
Koszul syzygies; see \cite[Tag 062D]{Stacks}. Thus there are
homogeneous $b_{ij}=-b_{ji}$ of weight $D-d_i-d_j$, with $b_{ii}=0$,
such that
$a_i^{[D-d_i]}=\sum_jb_{ij}q_j$.
Replace $a_i$ by $a_i-\sum_jb_{ij}f_j$.
The represented polynomial $F$ is unchanged, because the sum
$\sum_{i,j}b_{ij}f_jf_i$ is zero, and the new maximum weight is
strictly less than $D$. Descending through nonnegative integers
eventually gives a representation with $D\leq\deg F$.
Its highest part lies in $(q_i)$. This proves the nontrivial inclusion
in \eqref{eq:initial}; the reverse inclusion is immediate.

Now suppose the latter ideal is radical and $g^N\in(f_i)$ for some
$N\geq1$. The equality of initial ideals gives
$\htop(g)^N\in(q_i)$, and hence $\htop(g)\in(q_i)$.
Choose a homogeneous expression for $\htop(g)$ in the $q_i$ and
replace $q_i$ by $f_i$. This produces $F\in(f_i)$ with
$\htop(F)=\htop(g)$. The polynomial $g-F$ has smaller weight and
represents the same nilpotent residue class. Repetition terminates;
at weight zero the subtraction leaves zero. Hence $g\in(f_i)$.
\end{proof}

\begin{lemma}\label{lem:squarefree}
Let $L$ be a field, and let $S\in L[z_1,\ldots,z_e,x,y,w]$.
If its highest total degree part in $x,y,w$ is $cxyw$ for some
$c\in L^\times$, then $S$ is squarefree.
\end{lemma}

\begin{proof}
If a nonconstant irreducible polynomial $H$ occurred at least twice
in $S$, taking highest parts in $x,y,w$ would show that
$\htop(H)^2$ divides $cxyw$ in the polynomial ring over $L$.
Its only square divisors are units. This also excludes a factor
depending only on the $z$ variables, since such a factor would
divide the constant coefficient $c$. Thus $H$ cannot exist.
\end{proof}

\section{Reduced complete intersections}\label{sec:radical}

\begin{proof}[Proof of Theorem~\ref{thm:radical-intro}]
We first work over an algebraically closed field. The case $k=1$
is the ideal of the two entries of a $1\times2$ matrix. Suppose
$k\geq2$ and use induction on $k$.
Theorem~\ref{thm:dimension} gives height $k+1$ for $I_k(k,k+1)$.
Since the ideal has $k+1$ generators in a polynomial ring, these
generators form a regular sequence. Its quotient is a complete
intersection, Cohen--Macaulay, and pure of dimension
\begin{equation}\label{eq:global-CI-dimension}
 k(k+1)-(k+1)=k^2-1.
\end{equation}
We use the standard regular-sequence and unmixedness properties of
Cohen--Macaulay rings; see \cite[Tag 00N7]{Stacks}.

On a normalized nonzero-entry chart, write
\[
 X=\begin{pmatrix}Y&u\\ v^{\mathsf T}&1\end{pmatrix},
 \quad Y\text{ of size }(k-1)\times k,
 \quad u\in\A^{k-1},\quad v\in\A^k.
\]
For $j\in\{1,\ldots,k\}$ let $p_j$ be the permanent of $Y$ with
column $j$ omitted, and let $P_j$ be the maximal permanent of $X$
with that same column omitted. Let $R$ be the maximal permanent
omitting the pivot column. Set
\begin{equation}\label{eq:S}
 S=\sum_{i=1}^{k-1}\ \sum_{1\leq j<\ell\leq k}
 u_i v_jv_\ell\,
 \per Y[\{1,\ldots,k-1\}\setminus\{i\},
        \{1,\ldots,k\}\setminus\{j,\ell\}].
\end{equation}
Lemma~\ref{lem:cubic} gives
\begin{equation}\label{eq:S-syzygy}
 2S=\sum_{j=1}^k v_jP_j-R.
\end{equation}
Thus the normalized chart ideal is $(P_1,\ldots,P_k,S)$.
For the weights $\deg Y=1$, $\deg u=\deg v=0$, these generators
have highest parts $p_1,\ldots,p_k,S$, respectively.

By induction, the base ring
\[
 A=\kk[Y]/(p_1,\ldots,p_k)
\]
is reduced and a complete intersection, pure of dimension
$d=k(k-2)$. Write $B=\Spec A$.
For $k\geq3$, let $E\subset B$ be the locus where all
$(k-2)\times(k-2)$ permanents of $Y$ vanish. It is a subvariety of
$B$ by expansion of the $(k-1)$-permanents. Theorem~\ref{thm:dimension}
gives
\begin{equation}\label{eq:E-dimension}
 \dim E=k(k-3)=d-k.
\end{equation}
For $k=2$ take $E=\varnothing$ and use the empty permanent as a
nonzero core. In every case $E$ contains no component of $B$.

At a point of $B\setminus E$, choose a nonzero $(k-2)$-permanent
$c$ of $Y$, with row set $R_0$ and column set $T_0$.
There is a unique row $i$ outside $R_0$ and exactly two columns
$j,\ell$ outside $T_0$. Regarding the other border coordinates
as variables of weight zero, the highest part of $S$ in
$u_i,v_j,v_\ell$ is exactly
\begin{equation}\label{eq:radical-top}
 c\,u_i v_jv_\ell.
\end{equation}
This follows term by term from \eqref{eq:S}, just as in
\eqref{eq:second-top}.

Consequently $S$ is nonzero over the function field of every
irreducible component of $B$. In the reduced ring $A[u,v]$, an
element avoiding all minimal primes is a nonzerodivisor. Hence
\[
 J=(p_1,\ldots,p_k,S)\subset\kk[Y,u,v]
\]
is a complete-intersection ideal, and its quotient is pure of
dimension
\begin{equation}\label{eq:J-dimension}
 d+(2k-1)-1=k^2-2.
\end{equation}
By Lemma~\ref{lem:squarefree} and \eqref{eq:radical-top}, the
polynomial $S$ is squarefree over each of the base function fields.
Thus the generic fiber above every component of $B$ is reduced.

We must also exclude components that do not dominate a component of
$B$. Above every point of $B\setminus E$, $S$ is a nonzero
homogeneous cubic in $2k-1$ border variables. Its zero locus has
dimension exactly $2k-2$. Therefore the inverse image of a proper
closed subset of $B$ that contains no component, away from $E$,
has dimension at most
\[
 (d-1)+(2k-2)=k^2-3.
\]
Above $E$, even allowing the entire border space gives, for $k\geq3$,
\begin{equation}\label{eq:no-vertical}
 \dim E+2k-1=k^2-k-1<k^2-2.
\end{equation}
These estimates and purity \eqref{eq:J-dimension} show that every
irreducible component of $V(J)$ dominates a component of $B$.
It follows from the reduced generic fibers that $\kk[Y,u,v]/J$
is generically reduced. A Cohen--Macaulay ring has no embedded
associated primes, so generic reducedness implies reducedness;
equivalently, use the $R_0$ and $S_1$ criterion
\cite[Tag 031R]{Stacks}. Thus $J$ is radical.

Apply Lemma~\ref{lem:filtered} to
$(P_1,\ldots,P_k,S)$ and $(p_1,\ldots,p_k,S)$.
Every normalized nonzero-entry chart is reduced. Adjoining the
invertible pivot coordinate preserves reducedness, so the original
scheme is reduced away from the origin. Its complete-intersection
ring has no embedded primes, and every minimal component has positive
dimension $k^2-1$, so every such component meets a nonzero-entry
chart. It is therefore generically reduced, hence reduced everywhere.

Finally, the argument applies over an algebraic closure of every
field extension of the original ground field. Faithful flatness of
field extensions descends reducedness, while the height and
regular-sequence statements have already been proved over arbitrary
fields. This proves geometric reducedness and completes the induction.
\end{proof}

\section{Consequences and scope}\label{sec:consequences}

\begin{proof}[Proof of Corollary~\ref{cor:critical-intro}]
The partial derivative of $\per_n$ with respect to $x_{ij}$ is the
permanent of the complementary $(n-1)\times(n-1)$ matrix.
These complementary matrices comprise all submatrices of that size.
The critical locus is therefore $V_{n-1}(n,n)$, whose dimension is
$n(n-2)$ by Theorem~\ref{thm:dimension}.
For each fixed row $i$, expansion along that row gives
\[
 \per_n=\sum_{j=1}^n x_{ij}\frac{\partial\per_n}{\partial x_{ij}}.
\]
Thus the permanent belongs to its gradient ideal, and the critical
locus is also the singular locus of its zero hypersurface. This
argument does not divide by $n$ and applies when the characteristic
divides the matrix order.
\end{proof}

\begin{corollary}\label{cor:hilbert}
With the standard grading, the quotient by $I_k(k,k+1)$ has Hilbert
series
\[
 H(t)=\frac{(1-t^k)^{k+1}}{(1-t)^{k(k+1)}}.
\]
Its multiplicity is $k^{k+1}$, and its affine scheme is reduced
and pure of dimension $k^2-1$.
\end{corollary}

\begin{proof}
The $k+1$ generators form a homogeneous regular sequence, each of
degree $k$. The Hilbert-series formula and multiplicity follow by
successively quotienting by those nonzerodivisors. Reducedness and
purity were proved above. For $k=1$ the quotient is the ground field
and the formula gives $H(t)=1$.
\end{proof}

For clarity, define the \emph{product count} of a nonzero homogeneous
polynomial $f$ of degree at least two as the least $s$ for which
\[
 f=\sum_{i=1}^s A_iB_i,
\]
where $A_i,B_i$ are homogeneous of positive degrees adding to
$\deg f$. This is the convention for strength that counts products;
the convention with $s+1$ summands shifts the value by one.

\begin{corollary}\label{cor:strength}
Over every field of characteristic different from two, the product
count of $\per_n$ is exactly $n$ for $n\geq2$.
\end{corollary}

\begin{proof}
Any such expression implies
$V(A_1,B_1,\ldots,A_s,B_s)\subseteq\Crit(f)$ by differentiation.
The locus on the left is nonempty, since all factors have positive
degree, and its codimension is at most $2s$ by the height theorem.
For $f=\per_n$, Corollary~\ref{cor:critical-intro} therefore gives
$2n\leq2s$. Expansion along one row expresses $\per_n$ as $n$
products of the required kind. The lower bound can be checked after
extension to an algebraic closure and so holds over the original
field as well.
\end{proof}

\begin{remark}[Characteristic two]
The assumption on the characteristic is essential. In characteristic
two, permanents equal determinants. For example, the $2\times2$
permanents of a $3\times3$ matrix cut out the matrices of rank at
most one, a variety of dimension five rather than three.
In \eqref{eq:cubic-identity}, the coefficient two vanishes, so the
additional border equations cannot be deduced. For determinants in
other characteristics, the two ordered contributions in that
calculation have opposite signs and cancel.
\end{remark}

Theorem~\ref{thm:dimension} determines dimension, but does not
classify components or assert equidimensionality for every
$I_h(a,b)$. The reducedness theorem is restricted to
$I_k(k,k+1)$. Neither theorem implies Frobenius purity of all these
rings. The results also do not provide a superpolynomial circuit
lower bound: the product-count consequence above is linear in $n$.

\appendix
\section{Exact finite checks}\label{app:checks}

The accompanying program \texttt{verify.py} verifies identities by
expanding permanents as sums over permutations in a sparse polynomial
ring over $\mathbb Z$. It checks the bordered expansion,
\eqref{eq:cubic-identity} and its transpose, and the two types of
highest parts used in the proofs, for sizes two through five.
It also checks every coordinate support in the monomial model
\eqref{eq:support} for $2\leq k\leq8$.
The latter enumeration comprises $87{,}376$ pairs of supports.

The calculations use exact integer coefficients and no probabilistic
substitutions. They check the finite identities and support cases
specified by the program. They do not establish dimension, radicality,
or the assertions for arbitrary size by finite inference; those
claims depend on the proofs in the preceding sections. The program
uses only the Python standard library, and the supplement includes
its expected output and instructions for reproducing the checks.

\paragraph{Acknowledgment.}
The author acknowledges OpenAI GPT-6 for assistance with research
exploration, proof development, verification code, and manuscript
preparation.

\end{document}